\documentclass[10pt]{amsart}   
\usepackage{graphicx, amsmath, amssymb, amsthm, amscd}
\usepackage{color}
\usepackage{epsf}
\usepackage{enumerate}

\newtheorem{theorem}{Theorem}[section]
\newtheorem{lemma}[theorem]{Lemma}
\newtheorem{proposition}[theorem]{Proposition}

\newtheorem{cor}[theorem]{Corollary}
\theoremstyle{definition}
\newtheorem{question}{Question}
\newtheorem*{question*}{Question}
\newtheorem{conjecture}{Conjecture}
\newtheorem*{conjecture*}{Conjecture}

\newtheorem{definition}[theorem]{Definition}
\theoremstyle{remark}
\newtheorem{Remark}[theorem]{Remark}
\numberwithin{equation}{section}
\newtheorem{prop}[theorem]{Proposition}

 \newcommand{\Z}{\mathbb{Z}}

 \newcommand{\E}{\mathcal{E}}
 \newcommand{\F}{\mathcal{F}}

 \newcommand{\rs}{\mathfrak{s}}
 \newcommand{\Ann}{\mathbb{A}}
 
 \renewcommand{\Bbb}[1]{\mathbb{#1}}

\DeclareMathOperator{\Int}{int}
\DeclareMathOperator{\id}{id}
\DeclareMathOperator{\diam}{diam}
\DeclareMathOperator{\mesh}{\text{mesh}}

\begin{document}

\title[Minimal Map on pseudo-circle]{Minimal non-invertible maps on the pseudo-circle}

\author{Jan P. Boro\'nski}
\address[J. P. Boro\'nski]{AGH University of Science and Technology, Faculty of Applied
	Mathematics, al.
	Mickiewicza 30, 30-059 Krak\'ow, Poland
	-- and --
	National Supercomputing Centre IT4Innovations, Division of the University of Ostrava,
	Institute for Research and Applications of Fuzzy Modelling,
	30. dubna 22, 70103 Ostrava,
	Czech Republic}
\email{jboronski@wms.mat.agh.edu.pl}

\author{Judy Kennedy}
\address[Judy Kennedy]{Department of Mathematics, PO Box 10047, Lamar University, Beaumont, TX 77710, USA}
\email{kennedy9905@gmail.com}

\author{Xiao-Chuan Liu}
\address[Xiao-Chuan Liu]{Instituto de Matem\' atica e Estat\' istica da Universidade de S\~ao Paulo,
R. do Mat\~ ao, 1010 - Vila Universitaria, S\~ ao Paulo, Brasil} 
\email{lxc1984@gmail.com}

\author{Piotr Oprocha}
\address[P. Oprocha]{AGH University of Science and Technology, Faculty of Applied
Mathematics, al.
Mickiewicza 30, 30-059 Krak\'ow, Poland
-- and --
National Supercomputing Centre IT4Innovations, Division of the University of Ostrava,
Institute for Research and Applications of Fuzzy Modeling,
30. dubna 22, 70103 Ostrava,
Czech Republic}
\email{oprocha@agh.edu.pl}



\begin{abstract}
 In this article we show that R.H. Bing's pseudo-circle admits a minimal non-invertible map. 
 This resolves a conjecture raised  
  by Bruin, Kolyada and Snoha in the negative.
  The main tool is a variant of
  the Denjoy-Rees technique, further developed by 
 B\'eguin-Crovisier-Le Roux, combined with detailed study of the structure of 
 the pseudo-circle.  This is the first example of a planar 1-dimensional space that admits both minimal homeomorphisms and minimal noninvertible maps. 
\end{abstract}
\maketitle
\section{Introduction}
In the late 1960s J. Auslander raised a question concerning the existence of minimal noninvertible maps. Since then examples of such maps have been given, but the question as for what spaces such maps exist remains particularly interesting for compact and connected spaces (continua). In the present paper we are going to give the first example of a 1-dimensional space in the plane that admits both minimal homeomorphisms and minimal noninvertible maps. In general, for compact spaces the existence of an invertible minimal map is independent of the existence of a noninvertible one, and there exist spaces that admit both types of maps, either one of them, or none \cite{KST}. In \cite{AY} Auslander and Yorke observed that the Cantor set admits minimal maps of both kinds. In \cite{KST} Kolyada, Snoha and Trofimchuk showed how to obtain minimal noninvertible maps on the 2-torus from a class of minimal homeomorphisms. An analogous result was obtained in \cite{ST} by Sotola and Trofimchuk for the Klein bottle. In \cite{BCO} Clark, the first and the last mentioned authors showed the existence of both types of minimal maps for any compact finite-dimensional metric space that admits an aperiodic minimal flow. The first author showed the existence of minimal noninvertible maps for all compact manifolds that admit minimal homeomorphisms in \cite{B}. Auslander and Katznelson showed that the circle admits no minimal noninvertible maps, despite supporting minimal homeomorphisms \cite{AuK}. In this context Bruin, Kolyada and Snoha asked the following question:
\begin{question}\cite{BKS}
	Is the circle the only infinite continuum that admits a minimal homeomorphism but no noninvertible minimal map?
\end{question}
In the same paper Bruin, Kolyada and Snoha conjectured that R.H. Bing's pseudocircle might provide a counterexample (see also \cite{LS}). 
\begin{conjecture}\cite{BKS}
	The pseudo-circle is another continuum that admits minimal homeomorphisms, but no noninvertible minimal maps.
\end{conjecture}

Quite recently, in \cite{DST} a construction of a family of continua that admit minimal homeomorphisms, but no minimal noninvertible maps was given, resolving Question 1. However, Conjecture 1 has remained open until now. The pseudocircle seems indeed a natural candidate for another counterexample to Question 1, as it shares several properties with the circle.  Topologically it is a circle-like planar cofrontier, whose homeomorphism group contains periodic rotations \cite{Judy}, and it appears in various contexts in dynamical systems. Handel showed that it appears as both a minimal subset of a volume preserving, $C^\infty$-smooth planar diffeomorphism \cite{Handel}, and a minimal attractor of a $C^\infty$-smooth planar diffeomorphism. Herman \cite{He} extended Handel's construction to show that the pseudo-circle can bound an open disk in the complex plane, on which a smooth planar diffeomorphism is complex analytic and complex analytically conjugate to an irrational rotation, and on the complement of which the diffeomorphism is smoothly conjugate to such a rotation. This result was further improved by Ch\'eritat who showed that the pseudo-circle can appear as the boundary of a Siegel disk \cite{Che}. In \cite{BO} the first and last author showed that it also appears as a Birkhoff-like attractor \cite{BO}. Other related constructions were given by Yorke and the second author in  \cite{KY1}, \cite{KY2}, and \cite{KY3}. Most recently, a continuous decomposition of the 2-torus into pseudo-circles was given by  B\'eguin,  Crovisier, and Jäger, that is invariant under a homeomorphism of the torus semi-conjugate to an irrational circle rotation. In the present paper, relying on Handel's construction, we adapt and modify the results of B\'eguin,  Crovisier and Le Roux \cite{Cro} to the very fine and complex structure of the pseudo-circle, to provide a counterexample to the aforementioned conjecture. 
\begin{theorem}\label{main}
	The pseudo-circle admits minimal noninvertible maps. 
\end{theorem}
The results of \cite{Cro}, that are discussed in more detail in Section 3, are a far-reaching generalization of the Denjoy-Rees technique \cite{Denjoy}, \cite{Rees} that allowed Rees to construct the first examples of minimal torus homeomorphisms with positive entropy. These results are crucial to our construction. However, it should be noted that our proof is not a mere application of very general results of \cite{Cro} for manifolds to a special case of a particular 1-dimensional pathological fractal. Quite a lot of care and effort is needed to make sure that the construction will not change the underlying space, that is the pseudo-circle, into a different one-dimensional space. A naive approach would be to blow up an orbit of Handel's minimal homeomorphism contained in the pseudo-circle to a null sequence of pseudo-arcs, and then use this new homeomorphism with an asymptotic pair to obtain a minimal noninvertible map. In this approach one could think, in particular, of a straightforward application of the following proposition from \cite{Cro}. 
\begin{proposition}\label{Crop}(B\'eguin,Crovisier\&Le Roux, \cite{Cro})
	Let $R$ be a homeomorphism on a compact manifold $M$, and $x$ a point of $M$ which is not periodic under $R$. Consider a compact subset $D$ of $M$ which can be written as the intersection of a strictly decreasing sequence $\{D_n:n\in\mathbb{N}\}$ of tamely embedded topological closed balls. Then there exist a homeomorphism $f: M\to M$ and a continuous onto map $\phi : M \to M $ such that $\phi \circ f = R \circ \phi$, and such that
	\begin{itemize} 
		\item $\phi^{-1}(x)=D$;
		\item $\phi^{-1}(y)$ is a single point if $y$ does not belong to the $R$-orbit of $x$.
	\end{itemize}
\end{proposition}
However, this is not sufficient to obtain our result. The problem, again, is that 'inserting' the pseudo-arcs \textit{ad hoc} does not guarantee that the resulting space will not change, as there is even no guarantee that such an insertion would allow these pseudo-arcs to have empty interior in the resulting space, as is required for the pseudo-circle. Moreover, even if one can ensure that the special nontrivial fibres indeed have empty interior in the new space, one still has to ensure that the insertion is such that the resulting space remains circle-like (i.e. it can be covered by circular chains of arbitrarily small mesh; see Section 2 for precise definitions). Instead, we were lead to an adjustment of the method in \cite{Cro} (see the comments preceding Lemma \ref{semiconjugacy} for more details) and combine it with a careful analysis of the hereditary indecomposable structure of the pseudo-circle. To that end, the notion of an internal composant introduced by Krasinkiewicz in \cite{Kras} came in handy (see Section 4). Nevertheless, our construction made us appreciate the Denjoy-Rees technique, as presented in \cite{Cro}, even more and there are clearly other novel applications of it to come.

\noindent{\bf{Acknowledgements}.}
The authors are grateful to Fr{\'e}d{\'e}ric Le Roux for his comments on the constructions in \cite{Cro}.

J. P. Boro\'nski was supported by National Science Centre, Poland (NCN), grant no. 2015/19/D/ST1/01184. 
P. Oprocha was partially supported by the Faculty of Applied Mathematics AGH UST statutory tasks within subsidy of Ministry of Science and Higher Education and by NPU II project LQ1602 IT4Innovations excellence
in science  and project lRP201824 ``Complex topological structures''.
X.Liu is supported by 
Fapesp P\'os-Doutorado grant (Grant Number
 2018/03762-2). 
This work started when Liu was 
visiting Krakow in September, 2017. Part of this work was done when 
 he visited Paris in November, 2017, 
  supported by 
  Fondation Louis D-Institut de France, a 
  project coordinated by Marcelo Viana. \\

\section{Preliminaries and Some Previous Constructions}
In this section
we define the notation and terms we use as well as recall
some results from the literature that are
 used in the arguments. 
\subsection{Basic Notation in Topology and Dynamics}
For any planar set $E$, we denote by 
$\text{Int}(E)$ the interior of the set $E$. 
We work with the closed annulus $\Bbb A=S^1\times [0,1]$, and its
lift to $\widetilde{\Bbb A}=\Bbb R\times [0,1]$, with $\pi$ denoting the lift function from $\widetilde{\Bbb A}$ to $\Bbb A$.

 \vphantom{}

Given a family $\E$ of connected subsets of $\Bbb A$, 
define
\begin{equation}
\mesh (\E)=\max \{  \text{diam} (X) \big | X\in \E\}.
\end{equation}
For any family $\E$ of subsets of $\Bbb A$, 
denote by 
$\rs(\E)$ 
the union of all the elements of 
$\E$, called the \textit{realization} of $\mathcal E$.
A \textit{non-degenerate}
\textit{continuum} is a compact, connected metric space, which 
contains at least two elements. 
We call a continuum \textit{planar} if it can be embedded into the plane. 
A continuum $X$ is called  \textit{indecomposable}  
if $X$ is not the union of two proper subcontinua. 
A continuum is \textit{hereditarily indecomposable} 
if every subcontinuum of it is also indecomposable.
For any $x\in X$, the \textit{composant} of $x$, denoted  $C(x)$, is the union of all proper subcontinua in $X$ containing $x$.
In any indecomposable continuum each composant is a dense first-category connected set in $X$, and if $x,y$ are points in $X$, then either  $C(x)=C(y)$ or  $C(x)\cap C(y) = \emptyset$.
\begin{definition} 
Let $X$ be a planar continuum. 
A composant $C$ of $X$
 is called \textit{internal} 
if every continuum $L$ intersecting both $C$ and 
the complement of $X$ intersects all composants of $X$. 
\end{definition}

We will need
the following theorem proved by Krasinkiewicz (see Main Theorems of~\cite{Kras} and~\cite{Kras_01}).
\begin{lemma}\label{Internal_Composant}
For any indecomposable planar continuum $X$, 
the union of all the internal composants is a $G_\delta$ subset of 
$X$. In particular, every indecomposable planar continuum contains an internal composant.
\end{lemma}


Let 
$f\colon X\to X, 
g\colon Y\to Y$ be two homeomorphisms
on compact metric spaces 
$X$ and $Y$, respectively. 
Let $\psi\colon X \to Y$ be a continuous surjective map 
satisfying $\psi \circ f = g \circ \psi$. 
In this case, we say 
$\psi$ is a \emph{semi-conjugacy} between 
$(X,f)$ and $(Y,g)$. 
The map $\psi: X\to Y$ 
is called \textit{monotone}
if for any $y\in Y$, $\psi^{-1}(y)$ is connected.
We say the semi-conjugacy 
$\psi$ is \emph{almost $1$-$1$} 
if there exists a residual subset $Y_1\subset Y$,
 such that for each $y\in Y_1$, $\phi^{-1}(y)$
is a singleton. 

\vphantom{}

Given a monotone map 
$\psi: X\to Y$, the partition space 
$\{\psi^{-1}(y) \big| y \in Y\}$ is known to be \textit{upper semi-continuous}, i.e.,  each set $\psi^{-1}(y)$ is closed in $X$, 
and for any open set 
$U$ in $X$, the union of all $\psi^{-1}(y)$ contained in $U$ is open.  
We need the following 
well known properties for upper semi-continuous decompositions.
(See for example
Proposition 1 and Proposition 2 in chapter 1 
of~\cite{Daverman}.)

\begin{lemma}\label{usc}
Let $\phi: X\to Y$ be such that, 
$\{\phi^{-1}(y)\big| y\in Y\}$ forms 
an upper semi-continuous decomposition of $X$. 
Then the following hold: \begin{enumerate}
\item For any open set $U$ containing some pre-image $\phi^{-1}(y_0)$,
the set $W= \bigcup_{y\in \phi(U)}\phi^{-1}(y)$
is open. 
\item Suppose 
a sequence of closed 
sets $\phi^{-1}(y_n) \subset X$ 
converges to $A\subset X$ 
in the Hausdorff topology. 
Then for any $x\in A$, $A\subset \phi^{-1} \circ \phi(x)$.
\end{enumerate}
\end{lemma}

A \textit{Moore decomposition} of the closed annulus $\Bbb A$ 
is an upper semi-continuous decomposition $\mathcal M$ such that any 
$M\in\mathcal M$ is contained in a topological disk. 
The following is the classic Moore's theorem. 
\begin{lemma}[Section 25, Theorem 1 in~\cite{Daverman}]\label{Moore_theorem}
For any Moore decomposition $\mathcal M$ on $\Bbb A$ , 
the partition space is again $\Bbb A$ . Moreover, 
there exists a continuous monotone surjection 
$\phi:\Bbb A \to \Bbb A$, such that the preimage of each point is 
precisely one element $M\in \mathcal M$. 
\end{lemma}

\subsection{Circular Chains and Pseudo-circles}
Let $\Bbb Z_m$ denote 
the finite additive Abelian group
 $\Bbb Z_m=\{\overline 0,\cdots, \overline{m-1}\}$. 
 For any two elements $\overline i,\overline j \in \Bbb Z_m$, we use the convenient 
 notation $\rho$ for their distance. More precisely,
 denote $\rho(\overline{i},\overline{j})=\min( |i-j|, |m-|i-j||)$.
\begin{definition}
 For $m\geq 3$, a \textit{circular chain} 
is a finite family of open subset 
$\mathcal{D}=\{D_{\overline\ell}\}_{\overline \ell \in \Bbb Z_m}$, not necessarily connected, 
such that
\begin{equation}
D_{\overline \ell} \cap D_{\overline k} \ne \emptyset \text{ if and only if } 
\rho (\overline k, \overline \ell)\leq 1.
\end{equation}
Each element $D_{\overline \ell}$ is called a link.
\end{definition}


\begin{definition}
Let $\mathcal{D}=\{D_{\overline \ell} : \overline \ell \in \Z_m\}$ and
$\mathcal{C}=\{C_{\overline \ell} : \overline \ell \in \Z_s\}$ 
be two circular chains. 
We say $\mathcal D$ is \textit{crooked inside} 
$\mathcal{C}$
if there exists a map $\pi \colon \Z_m\to \Z_s$
with the following properties. 
\begin{enumerate}[(i)]
	\item for any $\overline i \in \Bbb Z_m$,
	the closure of $D_{\overline i}$ is contained in $C_{\pi(\overline i)}$.
	\item with properly chosen circular orders,
	let 
	$\overline i< \overline j$ 
	be two elements in $\Bbb Z_m$,
	such that $\pi(\overline i) \leq \pi(\overline j)$. 
	Suppose for any $\overline i<\overline k< \overline j$,
	$\pi(\overline i) \leq \pi(\overline k) \leq \pi(\overline j)$. 
	Then there exist
	$\overline i<\overline u<\overline v<\overline j$,
	such that, 
	\begin{align}
	& \rho(\pi(\overline i),\pi(\overline v))\leq 1. \\
	& \rho(\pi(\overline j),\pi(\overline u))\leq 1.
	\end{align}
	\end{enumerate}
\end{definition}

 We call a continuum $X$ \textit{circle-like} if for any $\varepsilon>0$, 
 $X$ can be covered by a circular chain $\mathcal C$ with $\mesh(\mathcal C)<\varepsilon$.
   
  \begin{definition}[See~\cite{Bing}, and see~\cite{Fe} for the uniqueness part.] \label{defi_pseudo_circle}
  The \textit{pseudo-circle} is the unique circle-like plane separating
  continuum that can be covered by a 
 decreasing family of circular chains 
 $\{\mathcal{C}_{n}\}_{n \geq 1}$,
  such that $\mathcal{C}_{n+1}$ is crooked inside 
  $\mathcal{C}_{n}$ for each $n\geq 0$, 
  and $\mesh (\mathcal C_n)\to 0$ as $n$ tends to $\infty$.
  \end{definition}

   \begin{Remark}\label{pseudo_circle_characterization}
 Alternatively, the pseudo-circle can also be characterised as circle-like, 
 hereditarily indecomposable plane separating continuum which can be embedded 
  in any two dimensional manifold (see~\cite{BingEmbedding}). 
 \end{Remark}

  \begin{Remark}\label{pseudo_arc} A pseudo-arc is similarly defined
  replacing the objects 
  ``circular chain" and 
  ``circle-like continua"
   by 
  ``linear chain" and 
  ``arc-like continua", respectively. 
  For precise definitions, see for example the book~\cite{Continuum}. 
  Here we want to mention several important facts. 
  First, the pseudo-arc is also unique up to homeomorphism (see~\cite{Bing}).
 Then it follows from both definitions that
  any non-degenerate proper subcontinuum 
  of a pseudo-circle is a pseudo-arc. 
  Another useful fact is that, 
  any non-degenerate subcontinuum of a 
  pseudo-arc is also a pseudo-arc (This is the main theorem of~\cite{Moise}).
 \end{Remark}

We also need the following Theorem from~\cite{Judy}.
\begin{lemma}[Theorem 3 of~\cite{Judy}]\label{lemma_of_composant}
Let $f$ be a homeomorphism of the pseudo-circle $P$.  Let $C$ 
be a composant of $P$ 
and suppose $f(C)=C$. Then $f$ admits a fixed point in $C$.
\end{lemma}
\subsection{The Denjoy-Rees Technique: Settings.}\label{Rees_Conditions}
Denote by $\Bbb A= S^1\times [0,1]$
for the closed annulus. Let 
$f \colon \Bbb A\to \Bbb A$ 
be a homeomorphism which preserves both the
orientation and the boundaries.  The following definitions are adopted from \cite{Cro}.

\begin{definition}
For an integer $p\geq 1$, 
a finite family $\E$
of closed disks, contained in $S^1\times (0,1)$,
 is called \textit{$p$-iterable} if for any 
 $X,Y\in \E$ and  
integers $-p\leq k,s\leq p$, 
either  $f^k(X)=f^s(Y)$
or $f^k(X)\cap f^s(Y)=\emptyset$. 
\end{definition}
For any $p$-iterable family of 
closed disks $\E$ and any $0\leq n\leq p$,
we denote
\begin{equation}
\E^n=\bigcup_{|k|\leq n} f^k(\E),
\end{equation}
where 
$f(\E)=\{f(X) \big | X\in \E\}$. 
For any $0\leq n < p$,
define an oriented graph $G(\E^n)$, 
where the vertices
 are elements of $\E^n$,
 and there is an edge from $X$ to $Y$
if and only if $f(X)=Y$. 
For a map $h$ such that $h(X)\subset X$ and $h(Y)\subset Y$, we say that it commutes with $f$ along edge $X\to Y$
if $f\circ h(x)=h\circ f(x)$.


\begin{definition}
Let $\E,\F$ be two finite families 
of closed disks, all contained in the open set $S^1\times (0,1)$. 
We say $\F$ refines $\E$ if the following conditions hold.
\begin{enumerate}
\item every element of $\E$ contains at least one element of $\F$.
\item for any 
$X\in \E$, $Y\in \F$,
 either $X\cap Y=\emptyset$ or $Y\subset \Int (X)$.
\end{enumerate}
\end{definition}

\begin{definition}
Let $\E,\F$ be 
two families of closed disks.
For an integer $p\geq 0$,
we say 
 \textit{$\F$ is compatible with $\E$ for $p$ iterates},
 if the following conditions hold.
 \begin{enumerate}
 \item $\E$ is $p$-iterable.
 \item  $\F$ is $(p+1)$-iterable.
 \item  $\rs (\F) \subset \rs (\E)$, and $\F^{p+1}$ refines $\E^p$. 
 \item  For every $k$ with $|k|\leq p$, 
  \begin{equation}
  \rs(\F^{p+1}) \cap f^k( \rs(\E) ) = f^k( \rs(\F) ).
  \end{equation}
\end{enumerate}
\end{definition}

Let $\{\E_{(n)}\}_{n\geq 0}$ denote 
a sequence of families of closed disks. 
The following list of hypotheses 
was formulated in Section 2 of~\cite{Cro}.
\begin{enumerate}
\item [$\mathbf{A_1}$.] For every $n \geq 0$,
       \begin{enumerate}
	\item $\E_{(n)}$ is 
	$(n+1)$-iterable and the graph $G(\E_{(n)}^n)$ has no cycle.
	\item $\E^{n+1}_{(n)}$ refines $\E_{(m)}^{m+1}$ for any $0\leq m <n$;
	\item $\E_{(n)}$ is compatible with $\E_{(n-1)}$ for $n$ iterates.
       \end{enumerate}
\item [$\mathbf{A_2}$.]  The following holds. 
\begin{equation}
\lim_{n\to \infty}\mesh (\E_{(n)}^n) = 0.
\end{equation}
\end{enumerate}
Next, we want to choose a sequence of 
homeomorphisms $\{h_n\}_{n\geq 1}$, 
such that the following hold for every $n\geq 1$.
\begin{itemize}
    \item[$\mathbf{B_{1}}$.] 
the closure
$\overline{\{x : h_n(x)\neq x\}}$ is contained
in the set $\rs (\E_{(n-1)}^{n-1})$.
    \item[$\mathbf{B_{2}}$.] 
$h_n$ and $f$ commute 
along edges of the graph $G(\E_{(n-1)}^{n-1})$.
    \item[$\mathbf{B_3}$.] the following holds. 
\begin{equation}
\lim_{n\to \infty}
 \mesh \big(h_1^{-1}\circ \cdots \circ h_{n-1}^{-1}
 ( \E_{(n)}^{n+1} \backslash \E_{(n)}^{n-1}) \big)=0.
\end{equation}
\end{itemize}
Whenever we have such a sequence of homeomorphisms $h_n$, for any $n\geq 1$,
we define the homeomorphisms 
$\psi_n, g_n$ as follows. 
\begin{align}
\psi_n &=h_n \circ \ldots \circ h_1.\\
g_n &=\psi_n^{-1}\circ f \circ \psi_n.
\end{align}
Finally we set 
$\psi_0=\id$ and $g_0=f$. 

\subsection{The Denjoy-Rees Technique: Results.}

The following results 
proved in \cite{Cro} 
form the starting point of 
our proof. 
They ensure 
proper convergence 
of the sequence of conjugated homeomorphisms, 
while the conjugacies
converge to a continuous map, which provides a semi-conjugacy.
Later we will have to adjust this approach, when defining maps $h_n$ to ensure sufficiently fast convergence. This will be possible due to different approach
to sets $X_n$ in \eqref{eq:211}. For purposes of \cite{Cro} sets $X_n$ were specified as the beginning, and then the sequence was refined sufficiently fast. Unfortunately, this approach
is likely to lead to a situation as on Fig.~\ref{fig:chains}, destroying chainablity or other properties. To deal with this problem, sets $X_n$ will be defined dynamically as the construction progresses (and so the continuum in the fiber is unknown to the very end of the construction).

\begin{figure}[h!]
	\includegraphics[width=0.6\textwidth]{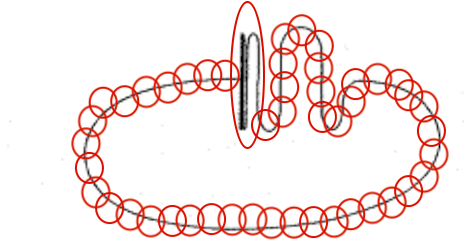}
	\caption{Fiber nowhere dense in $\mathbb{T}^2$ but large in ``blown up'' continuum.}\label{fig:chains}.
\end{figure}

\begin{lemma}[Proposition~3.1 in~\cite{Cro}]\label{lem:conv}
Assume that hypotheses $\mathbf{A_{1}}, \mathbf{A_2}, \mathbf{B_{1}}, \mathbf{B_2}, \mathbf{B_3}$ are satisfied. Then:
\begin{enumerate}
	\item The sequence $\{\psi_n\}_{n\geq 0}$ 
	converges uniformly to a continuous surjective map 
	$\psi \colon \Bbb A\to \Bbb A$.
	\item 
	The sequences $\{g_n\}_{n\geq 1}$ 
	and $\{g_n^{-1}\}_{n\geq 1}$
	 converge uniformly to a homeomorphism  
	$g \colon \Bbb A\to \Bbb A$ 
	and its inverse $g^{-1}$, respectively.
	\item The homeomorphism $f$
	is semi-conjugate to $g$ via the map $\psi$. 
	\end{enumerate}
\end{lemma}

\begin{lemma}[Proposition~3.4 in~\cite{Cro}]\label{lem:fibres}
Assume that hypotheses $\mathbf{A_{1}},
 \mathbf{A_2}, \mathbf{B_{1}}, \mathbf{B_2}, \mathbf{B_3}$ are satisfied.
Let $K=\bigcap_{n \geq 0}\rs(\E_{(n)})$.
\begin{enumerate}
	\item Suppose $x\in \Bbb A$ 
	and  
	there exists $m\in \Z$ such that $x\in f^m(K)$. 
	Let $\{X_n\}_{n\geq |m|} \subset \E_{(n)}^{|m|}$ 
	denote the decreasing sequence of closed disks
	containing $x$.
	Then
	\begin{equation}
	\psi^{-1}(x)=\bigcap_{n\geq |m|}\psi_n^{-1}(X_n).\label{eq:211}
	\end{equation}
	\item For every $x$ which does not belong to the orbit of $K$,
	$\psi^{-1}(x)$ is a singleton.
\end{enumerate}
\end{lemma}

\section{Minimal non-invertible map on the pseudo-circle}

The main result of this section is the following 
Theorem. 

\vphantom{}

\begin{theorem}\label{main_section}
There exists a homeomorphism $f\colon \Ann\to \Ann$ 
with an invariant pseudo-circle $Q \subset \Bbb A$ 
such that $(Q,f)$ 
is minimal and there exists a  
pseudo-arc $A\subset Q$
such that 
\begin{equation}\label{shrinking}
\lim_{|n|\to \infty}\text{diam} \big(   f^n(A) \big)=0.
\end{equation}
\end{theorem}

The proof of Theorem~\ref{main_section} goes through the whole section. 
In fact we will prove more than the statement here. We will provide 
a semi-conjugacy between two annulus homeomorphisms.   
The conclusion of the theorem is based on the existence of such a semi-conjugacy.

Let us give an outline before getting into details. 
In Section~\ref{Handel_section}, we recall Handel's construction of an annulus 
homeomorphism $f$, 
which admits an invariant pseudo-circle $P\subset \Bbb A$. 
Then, we construct the semi-conjugacy in Section~\ref{semi} 
between $(\Bbb A,f)$ 
and $(\Bbb A,g)$ for some other homeomorphism $g$ on $\Bbb A$. 
This is the most delicate part of the construction.
In Section~\ref{Q_is_pseudo}, we show the pre-image of $P$ under the semi-conjugacy
is still a pseudo-circle, and, finally, the rest of the proof of Theorem~\ref{main_section}
is given in Section~\ref{end_minimal}.

Based on Theorem~\ref{main_section}, the proof of Theorem~\ref{main} is 
given in Section~\ref{final}.

\subsection{Minimal Homeomorphism on the Pseudo-circle}\label{Handel_section}
Here we define a homeomorphism $f$ on 
the closed annulus $\Bbb A$ which preserves a minimal 
invariant pseudo-circle $P$. 
The construction follows 
directly from
~\cite{Handel}, so 
we refer to this paper
for more details (and some pictures for the construction). 
Roughly speaking, the pseudo-circle in construction in \cite{Handel} is obtained as a nested sequence of annuli.
Annulus in $n$-th step of the construction 
is identified with $S^1\times [0,1]$ and then
circle $S^1$ is split into 
 exactly $m_n$ sub-intervals of equal lengths, say, $I^{(0)}_1,\cdots,I^{(0)}_{m_0}$.
This way we define circular chain $\mathcal C_n$, taking consecutive link
 of the form 
 $(I^{(0)}_k\cup I^{(0)}_{k+1}) 
 \times (0,1)$, for $k=1,\cdots, m_n-1$, 
 and putting as the last $(I^{(0)}_{m_n}\cup I^{(0)}_1) \times (0,1)$.
 This allows us to express the main result of \cite{Handel} in the following form.

\begin{lemma}\label{Handel} Let $\Bbb A=S^1\times [0,1]$.
There exists a sequence of circular chains $\{\mathcal C_n\}_{n\geq 0}$ with links whose closures are closed discs,
and a homeomorphism $f: \Bbb A\to \Bbb A$, with the following properties.
\begin{enumerate}
\item  The following conditions are satisfied. 
\begin{align}
      \overline {\rs (\mathcal C_0)} & \subset S^1 \times (0,1).\\
\overline{\rs (\mathcal C_{n+1})} & \subset \rs ( \mathcal C_n), \text{ for any }n\geq 0.\label{cn:subsets}
\end{align}
\item There is an $f$-invariant minimal pseudo-circle $P$
such that
\begin{equation}
P = \bigcap_{n\geq 0} \rs(\mathcal C_{n}) = \bigcap_{n\geq 0} \overline{\rs (\mathcal C_n)}.
\end{equation}
\end{enumerate}
\end{lemma}

\subsection{Construction of the Semi-conjugacy.}\label{semi}
In this subsection, we construct the semi-conjugacy. The main tool from 
this subsection is Lemma~\ref{semiconjugacy}.
\begin{lemma}\label{Ap} 
For the homeomorphism $f:\Bbb A\to \Bbb A$
constructed in Lemma~\ref{Handel}, 
there exists a sequence of pseudo-arcs 
$A_0\supset A_1 \supset \cdots$ contained in $P$,
such that, 
\begin{equation}
\lim_{n\to \infty} \text{diam}\big( A_n\big) =0.
\end{equation}
Moreover, for all $n\geq 0$,
each $A_{n+1}$ is contained in an internal 
composant of $A_n$.
In particular, $\bigcap_{n\geq 0} A_n$ is a singleton, 
containing a unique point, called $p_\ast$. 
\end{lemma}

\begin{proof}
By Lemma~\ref{Internal_Composant}, 
there exists an internal composant in $P$. 
So, 
we can choose $A_0\subset P$ to 
be a pseudo-arc in $P$, with diameter smaller than $2^{-1}$,
which 
is contained in an internal composant of $P$. 
Then, for the same reason, we can choose $A_1$ to 
be a pseudo-arc contained in some internal composant of $A_0$,
of diameter less than $2^{-2}$.
Inductively, 
we can choose a sequence of pseudo-arcs, namely $\{A_n\}_{n\geq 1}$, 
such that, for any $n\geq 0$,
the diameter $\text{diam}(A_n)$
is smaller than $2^{-n}$, 
and each $A_{n+1}$ is contained in 
an internal composant of $A_n$. The proof is completed. 
\end{proof}

Now we will perform a construction inspired by \cite{Cro}. For our proofs it is essential to have both conditions \eqref{Cn-inclusion-1} and \eqref{Cn-inclusion-2} below satisfied and have a pseudoarc in one of the fibres. To do it following a scheme from \cite{Cro} we would start with
a continuum $A\subset P$ that we want to have in a fibre and then we use a decreasing family of its neighborhoods $U_n$ such that $\cap_n U_n=A$.
However, if one checks the details carefully, then it becames evident that in this approach always one of the conditions  \eqref{Cn-inclusion-1} or \eqref{Cn-inclusion-2} will be broken and it is impossible to satisfy them simultanously.
Therefore we need an adjustment in the construction. The idea here is that we do not fix the seqence $U_n$ apriori, but construct it inductively ($U_{n+1}$ can be defined after $h_1,\ldots,h_n$ are known). In the process, we have to ensure that the sequence of connected open set $U_n$ is nested (i.e. $\overline{U_{n+1}}\subset U_n$) and continuum obtained at the end $A=\cap_n U_n$
is within the family of continua that we are interested in (in our case, the pseudo-arc). These are main issues where we have to be especially careful.

\begin{lemma}\label{semiconjugacy}
	Denote by $f, P$ and $p_\ast$ the objects obtained in Lemmas~\ref{Handel} and~\ref{Ap}.
	Then there exist homeomorphisms $h_n$ on $\mathbb{A}$ such that if we denote
	\begin{align}
	\psi_n: =h_n\circ \cdots\circ h_1.\\
	g_n= \psi_n^{-1}\circ f\circ \psi_n.
	\end{align}
	the the following limits exist (in the topology of uniform convergence)
	\begin{align}
	g: =\lim_{n\to\infty} g_n.\\
	\psi\colon= \lim_{n\to \infty}\psi_n
	\end{align}
	and dynamical system $(\mathbb{A},f)$ is a factor of $(\mathbb{A},g)$ by the factor map $\psi$. For every $x\in \mathbb{A}$ not in orbit of $p_\ast$ (i.e. $f^n(x)\neq p_\ast$ for every $n\in \Z$) the set $\psi^{-1}(x)$ is a singleton, $\psi^{-1}(p_\ast)$ is a pseudoarc and additionally there is an increasing sequence $k_\ell$ such that:
	\begin{align}
	h_{\ell+1}^{-1} \big( \overline{ \rs (\mathcal{C}_{k_{\ell+1}}) }\big) & \subset 
	\rs\big( \mathcal{C}_{k_{\ell}} \big).\label{Cn-inclusion-1}\\
	\mesh \big( (h_{\ell}\circ\cdots\circ h_1)^{-1} ( \mathcal{C}_{k_{\ell}} ) \big) & < 2^{-\ell}\label{Cn-inclusion-2}.
	\end{align}
	for every $\ell\geq 1$
\end{lemma}

\begin{proof}
	Let $p_\ast\in \cap_n A_n\subset P$ be pseudo-arcs provided by Lemma~\ref{Ap}.
	Recall that in the construction of $f$ in Lemma~\ref{Handel}, 
	we have chosen a sequence of circular chains $\{\mathcal C_n\}_{n\geq 0}$, 
	such that $P=\bigcap_{n\geq 0} \rs(\mathcal C_n)$. 
	Then, define $V_n$ to be an element contained in $\mathcal C_n$ which contains $p_\ast$. 
	We also requite (refining sequence $\mathcal{C}_n$ when necessary), that the elements of $\{f^i(\overline{V_n})\}_{i=-(n+2)}^{n+2}$ are 
	pairwise disjoint and the sequence $V_n$ is nested.

	As $U_0$ we take a neighborhood of the pseudo-arc $A_0$ such that $U_0\subset \rs(\mathcal C_0)$ and $\{f^i(\overline{U_0})\}_{i=-2}^{2}$
	has pairwise disjoint elements. Further sets $U_n$
	will be defined inductively.
	
	%
	%
	
	Regarding Lemmas~\ref{lem:conv} and~\ref{lem:fibres}, we define inductively 
	the following objects.
	\begin{enumerate}
		\item The sequence of families of closed disks, denoted by $\{\mathcal E_{(n)}\}_{n\geq 0}$, 
		satisfying conditions $\mathbf{A_1}$ and $\mathbf{A_2}$ in Section~\ref{Rees_Conditions}.
		\item The sequence of homeomorphisms 
		$\{h_n\}_{n\geq 1}$ on the annulus $\Bbb A$  
		satisfying conditions $\mathbf{B_1}, \mathbf{B_2}, \mathbf{B_3}$ in Section~\ref{Rees_Conditions}.
	\end{enumerate}
	We start by setting $\E_{(0)}=\{\overline{U_0}\}$, 
	and $k_0=0$. 
	It is immediate from definition that 
	$\E_{(0)}$ satisfies conditions of $\mathbf{A_1}$
	(only item (a) of $\mathbf{A_1}$ needs to be checked for $n=0$).
	Note that 
	\begin{equation}
	\E_{(0)}^1=\{f^{-1}(\overline{U_0}), \overline{U_0}, f(\overline{U_0})\}.
	\end{equation}
	Let $U\subset \overline{U}\subset U_0$ be another neighborhood of $A_0$, such that $\overline{U}$ is homeomorphic to a closed disc, and is contained in the union of links of a chain cover of $A_0$ by discs of diameter at most $2^{-1}$.
	In other words, we may treat $U$ as an open set and as a chain cover.
	Take any integers $s_1>0 $ such that $V_{s_1}\subset U$. 
	
	Then we define 
	the homeomorphism
	$h_1\colon \Bbb A\to \Bbb A$, with the following properties. 
	\begin{align}
	h_1 (U)  &  = V_{s_1}. \label{h1UV}\\
	h_1(p_\ast)      &  = p_\ast.\label{h1UV+}\\
	h_1 \text{ restricted to }&  \Bbb A\backslash 
	U_0 \text{ is the identity}.\label{support_ofh1}
	\end{align} 
	Note that 
	the condition $\mathbf{B_1}$ for $n=1$ follows from (\ref{support_ofh1}). 
	Condition 
	$\mathbf{B_2}$ is empty for $n=1$. Denote 
	$\gamma=\diam A_0/2$. Let $k_1>0$ be an integer large enough, so that 
	$$
	\mesh ( h_1^{-1} ( \mathcal{C}_{k_{1}} ) ) < 2^{-1}
	$$
	and put $U_{k_1}=U$.
	Define now 
	\begin{equation}
	\E_{(1)} : = \{\overline{V_{s_1}} \}.
	\end{equation}  
	By definition
	$\E_{(1)}$ is $(s_1+1)$-iterable, and $s_1\geq 1$.
	It also follows  that
	$\E_{(1)}^2$ refines $\E_{(0)}^1$, 
	and $\E_{(1)}$ is compatible with $\E_{(0)}$ for $1$ iterate.
	In other words, the three conditions in $\mathbf{A_1}$  
	are all satisfied for $n=1$. 
	
	Now we will proceed with main construction, which will be repeated inductively.
	Take any $m>k_1+s_1$ sufficiently large so that, if $R\ni p_\ast$ is a link of $\mathcal{C}_m$ then the sets $\{f^i(\overline{R}) : i=-3s_1,\ldots,3s_1\}$
	are pairwise disjoint and $\overline{R}\subset V_{s_1}$. Let $s$ be any integer such that $A_s\subset R$.
	
	Since in the definition of $h_{1}$ we only require \eqref{h1UV} and \eqref{h1UV+}, we can freely modify it on points of $U_{k_1}\supset A_0$, that is we may assume that $\diam h_1^{-1}(A_s)$ is arbitrarily close to $\diam U_{k_1}$, in particular exceeds $\gamma$. 
	
	Note that $A_s\subset P \subset \rs(\mathcal{C}_m)$. Let $m'>m$ be such that
	$$
	f^i(\rs(\mathcal{C}_{m'} ))\subset \rs (\mathcal{C}_{m}) \qquad \text{ for }|i|\leq 1,
	$$
	and
	$$\mesh h_{1}^{-1}(\mathcal{C}_{m'})<2^{-2}.$$
	We also assume that $h_{1}^{-1}(\mathcal{C}_{m'-1})$ defines a chain cover of $h_{1}^{-1}(A_s)$, which is a refinement of $U$, treated as a chain. Recall that $\mathcal{C}_{m'}$ is crooked inside $\mathcal{C}_{m'-1}$ by definition. Let $p_\ast\in W\subset \overline{W}\subset U_{k_1}$ be a neighborhood of $h_1^{-1}(A_s)$ obtained as the union of elements of a $2^{-2}$-chain cover of the pseudoarc $ h_1^{-1}(A_s)$ consisting of elements of $h_{1}^{-1}(\mathcal{C}_{m'})$ . Since links of $\mathcal{C}_{m'}$ are connected, $W$ is connected as well.
	
	If we view $U_{k_1}$ and $W$ as chains,
	then we may assume that $W$ is $1/4$-chain crooked inside $U_{k_1}$ (in particular, it intersects all links that $U_{k_1}$ is union of).
	Let $W'$ be an open neighborhood of $A_s$ such that $\overline{W'}$ is homeomorphic to a closed disc and
	$$
	A_s\subset W'\subset \overline{W'}\subset h_1(W)\subset V_{s_1}
	$$ and
	take any $s_2>m'+k_1+s_1$ such that
	$     \overline{V_{s_2}}  \subset W'$. 
	
	Recall that the elements of $\{f^i(\overline{V_{s_2}}): |i|\leq 2\}$ are pairwise disjoint sets.
	Then
	we can define 
	the homeomorphism $h_2$ 
	satisfying the following properties. 
	\begin{align}
	h_2(W') & =V_{s_2}, \label{vk2}\\
	h_2(p_\ast) & =p_\ast,\label{past2}\\
	h_2 \text{ commutes with } f & \text{ on } \bigcup_{|i|\leq 1} f^{i}(\overline{h_1(W)}),\label{commute21}\\
	h_2 \text{ restricted to }&  \Bbb A\backslash 
	\bigcup_{i=-1}^1 f^i(h_1(W)) \text{ is the identity}.\label{support_ofh21}
	\end{align}
	For \eqref{commute21}, 
	note that the sets 
	$f^{-1}( V_{s_1} ),  V_{s_1}$,
	and $f(V_{s_1})$
	are pairwise disjoint, so first we 
	define the restriction of $h_2$ to $\overline{W'}$ and then extend it to $V_{s_1}$ so that $h_2|_{(V_{s_1}\setminus h_1(W))}$ is the identity. Next put
	$f^i\circ h_2\big |_{V_{s_1}}\circ f^{-i}$ 
	for $|i|\leq 1$,  and finally put
	$h_{2}(x)=x$ for all other points $x\in \Bbb A$. 
	This way a homeomorphism $h_2$ on $\Bbb A$ is defined, 
	which satisfies 
	(\ref{commute21}) and (\ref{support_ofh21}).  
	
	Now take $k_2$ sufficiently large, so that
	\begin{align}
	h_{2}^{-1} \big( \overline{ \rs (\mathcal{C}_{k_{2}}) }\big) & \subset 
	\rs\big( \mathcal{C}_{k_{1}} \big),\label{3.25}\\
	\mesh (\big( h_{2}\circ h_1)^{-1} ( \mathcal{C}_{k_{2}} ) \big) & < 2^{-2},\label{3.26}
	\end{align}
	and put $U_{k_2}=h_1^{-1}(W')$. 
	
	As long as \eqref{3.26} is fairly easy to satisfy, \eqref{3.25} is possible since
	$$
	\{x: h_2(x)\neq x\}\subset \cup_{|i|\leq 1} f^i(h_1(W))\subset \cup_{|i|\leq 1} f^i(\rs(\mathcal{C}_{m'})) \subset \rs(\mathcal{C}_m) \subset \rs (\mathcal{C}_{k_1}),
	$$
	in particular $h_2^{-1}(P)\subset \rs (\mathcal{C}_{k_1})$.
	Note that $U_{k_2}$ is connected, so we also may view it as a chain with links induced by $W$.
	To finish the construction, put
	\begin{equation}
	\E_{(2)} : = \{\overline{V_{s_2}}\}.
	\end{equation}

	Directly from the definition of $h_2$,
	conditions $\mathbf{B_1}$ and $\mathbf{B_2}$ are satisfied for $n=2$.
	Sets $f^{i}(V_{s_2})$ are pairwise disjoint for $|i|\leq s_2$ and $s_2\geq 4$, so we obtain that 
	$\E_{(2)}$ is at least $4$-iterable. Since $V_{s_2}$ is a neighborhood of $p_\ast$, it is easy to satisfy the requirement that $\E_{(2)}^3$ refines $\E_{(1)}^2$, 
	and clearly also $\E_{(2)}$ is compatible with $\E_{(1)}$ for $2$ iterates.
	Summing up, the three conditions in $\mathbf{A_1}$ are easily satisfied for $n=2$.
	
	Assume now that for $i=1,2,\ldots,n$ we have already defined increasing sequences $s_i,k_i$, open sets $U_{k_i}$ and maps $h_i\colon \Bbb A\to\Bbb A$
	such that $h_i(p_\ast)=p_\ast$ and:
	\begin{enumerate}
		\item $U_{k_{i}}\subset \overline{U_{k_{i}}}\subset U_{k_{i-1}}$, $h_1\circ \ldots \circ h_i(U_{k_i})=V_{s_i}$
		and for some $s>s_n$ we have $A_s\subset V_{s_n}$ and $\diam (h_1\circ \ldots \circ h_n)^{-1}(A_s)>\gamma$.
		\item $\E_{(n)}=\{\overline{V_{s_n}}\}$ and conditions $\mathbf{A_1}$, $\mathbf{B_1}$, $\mathbf{B_2}$  and $\mathbf{B_3}$ are satisfied for each $i\leq n$.
		\item Each $\overline{U_{k_i}}$ is homeomorphic to a closed disc, we can present each set $U_{k_i}$ as a union of elements of a $1/2^i$-chain and under this representation $U_{k_{i}}$ is crooked in $U_{k_{i-1}}.$
		\item The following conditions are satisfied:
		\begin{align}
		h_{i-1}^{-1} \big( \overline{ \rs (\mathcal{C}_{k_{i}}) }\big) & \subset 
		\rs\big( \mathcal{C}_{k_{i-1}} \big). \\
		\mesh \big( (h_i\circ\cdots\circ h_1)^{-1} ( \mathcal{C}_{k_{i}} ) \big) & < 2^{-i}.
		\end{align}
		
	\end{enumerate}
	
	Take any $m>k_n+s_n$ sufficiently large so that, if $R\ni p_\ast$ is a link of $\mathcal{C}_m$ then sets $\{f^i(\overline{R}) : i=-3s_n,\ldots,3s_n\}$
	are pairwise disjoint and $\overline{R}\subset V_{s_n}$. Let $r$ be any integer such that $A_r\subset R$. The only requirements on $h_n$
	are that $h_n(p_\ast)=p_\ast$ and $h_n(h_{n-1}\circ\ldots\circ h_1(U_{k_n}))=V_{s_n}$. Furthermore $A_r\subset A_{s}\subset V_{s_n}$ and $\diam (h_n\circ\cdots\circ h_1)^{-1} (A_s)>\gamma$. Therefore, we may adjust definition of $h_n$ on the set $h_{n-1}\circ\ldots\circ h_{1}(U_{k_n})$
	in such a way that also $\diam (h_n\circ\cdots\circ h_1)^{-1} (A_r)>\gamma$ and all other assumed properties hold (in particular, we have to adjust it also on some images of $U_{k_n}$
	so that  commutativity in $\mathbf{B}_2$ remains satisfied).
	
	
	Note that $A_r\subset P \subset \rs(\mathcal{C}_m)$. Let $m'>m$ be such that
	$$
	f^i(\rs(\mathcal{C}_{m'} ))\subset \rs (\mathcal{C}_{m}) \qquad \text{ for }|i|\leq n,
	$$
and 
$$	\mesh \big( (h_n\circ\cdots\circ h_1)^{-1} ( \mathcal{C}_{m'} ) \big) < 2^{-n-1}.$$	
We also assume that $(h_n\circ\cdots \circ h_{1})^{-1}(\mathcal{C}_{m'-1})$ defines a chain cover of $(h_n\circ\cdots \circ h_{1})^{-1}(A_r)$, which is a refinement of $U_{k_n}$, treated as a chain. 	Let $p_\ast\in W\subset \overline{W}\subset U_{k_n}$ be a neighborhood of $(h_n\circ\cdots\circ h_1)^{-1} (A_r)$ obtained as the union of elements of a $2^{-n-1}$-chain cover of the pseudoarc $(h_n\circ\cdots\circ h_1)^{-1} (A_r)$ consisting of elements of $(h_n\circ\cdots\circ h_1)^{-1} (\mathcal{C}_{m'})$ . Since links of $\mathcal{C}_{m'}$ are connected, $W$ is connected as well.
	
	If we view $U_{k_n}$ and $W$ as chains,
	then we may assume that $W$ is $2^{-n-1}$-chain crooked inside $U_{k_n}$ (in particular intersects all links that $U_{k_n}$ is union of).
	Let $W'$ be an open neighborhood of $A_r$ such that $\overline{W'}$ is homeomorphic to a closed disc and
	$$
	A_r\subset W'\subset \overline{W'}\subset h_n\circ \ldots\circ h_1(W)\subset V_{s_n}
	$$ and
	take any $s_{n+1}>m'+k_n+s_n$ such that $\overline{V_{s_{n+1}}}  \subset W'$. 
	
	Recall that the elements of $\{f^i(\overline{V_{s_{n+1}}}): |i|\leq n+1\}$ are pairwise disjoint sets.
	Then
	we can define 
	the homeomorphism $h_{n+1}$ 
	satisfying the following properties. 
	\begin{align}
	h_{n+1}(W') & =V_{s_n}, \label{vkn1}\\
	h_{n+1}(p_\ast) & =p_\ast,\label{pastn}\\
	h_{n+1} \text{ commutes with } f & \text{ on } \bigcup_{|i|\leq n} f^{i}(\overline{h_n(W)}),\label{commute2}\\
	h_{n+1} \text{ restricted to }&  \Bbb A\backslash 
	\bigcup_{i=-n}^n f^i(h_n(W)) \text{ is the identity}.\label{support_ofh2}
	\end{align}
	For \eqref{commute2}, 
	note that the sets 
	$f^{i}( V_{s_n} )$, $|i|\leq n$
	are pairwise disjoint, so first we 
	define the restriction of $h_{n+1}$ to $h_n(W)\subset V_{s_n}$, extend it by identity to $V_{s_n}$, put
	$f^i\circ h_{n+1}\big |_{V_{s_n}}\circ f^{-i}$ 
	for $i=-n,\ldots,n$,  and finally put
	$h_{n+1}(x)=x$ for all other points $x\in \Bbb A$. 
	In this way a homeomorphism $h_{n+1}$ on $\Bbb A$, 
	which satisfies 
	(\ref{commute2}) is defined.  
	
	Now take $k_{n+1}$ sufficiently large, so that
	\begin{align}
	h_{n+1}^{-1} \big( \overline{ \rs (\mathcal{C}_{k_{n+1}}) }\big) & \subset 
	\rs\big( \mathcal{C}_{k_{n}} \big), \label{3.25-n}\\
	\mesh (\big( h_{n+1}\circ\ldots \circ h_1)^{-1} ( \mathcal{C}_{k_{n+1}} ) \big) & < 2^{-n-1}\label{3.26-n}
	\end{align}
	and put $U_{k_{n+1}}=(h_{n+1}\circ\ldots \circ h_1)^{-1}(W')$. Note that $U_{k_{n+1}}$ contains $(h_n\circ\cdots\circ h_1)^{-1} (A_r)$,
	in particular $\diam(U_{k_{n+1}})>\gamma$. 
	
	As before \eqref{3.26-n} is easy to satisfy and \eqref{3.25-n} is possible, because
	\begin{eqnarray*}
		\{x: h_{n+1}(x)\neq x\}&\subset& \cup_{|i|\leq n} f^i(h_n\circ \ldots\circ h_1(W))\\
		&\subset& \cup_{|i|\leq n} f^i(\rs(\mathcal{C}_{m'})) \subset \rs(\mathcal{C}_m) \subset \rs (\mathcal{C}_{k_n}),
	\end{eqnarray*}
	in particular $h_{n+1}^{-1}(P)\subset \rs (\mathcal{C}_{k_n})$.
	Note that $U_{k_{n+1}}$ is connected, so we also may view it as a chain with links induced by $W$.
	Additionally, modifying $h_{n+1}$ accordingly on $U_{k_{n+1}}$, we may assume that there is $\hat{s}>0$ such that $A_{\hat{s}}\subset V_{s_{n+1}}$
	such that $\diam(h_{n+1}\circ\cdots\circ h_1)^{-1} (A_{\hat{s}})>\gamma$.
	
	To finish the construction, put
	\begin{equation}
	\E_{(n+1)} : = \{\overline{V_{s_{n+1}}}\}.
	\end{equation}

	Directly from the definition of $h_{n+1}$,
	conditions $\mathbf{B_1}$ and $\mathbf{B_2}$ are satisfied for $n+1$.
	Sets $f^{i}(V_{s_{n+1}})$ are pairwise disjoint for $|i|\leq s_{n+1}$ and $s_{n+1}\geq 2^{n+1}$, we obtain that 
	$\E_{(n+1)}$ is at least $2^{n+1}$-iterable. Since $V_{s_{n+1}}$ is a neighborhood of $p_\ast$, it is easy to satisfy the requirement that $\E_{(n+1)}^{n+2}$ refines $\E_{(n)}^{n+1}$, 
	and clearly also $\E_{(n+1)}$ is compatible with $\E_{(n)}$ for $n+1$ iterates.
	Summing up, the three conditions in $\mathbf{A_1}$ are easily satisfied for $n+1$.
	This completes the induction.
	
	Since family of sets $V_n$ intersects in a single point, condition $\mathbf{A_2}$ is satisfied, proving all the properties claimed for the construction.

	Define for all $n\geq 1$,
	\begin{align}
	\psi_n: =h_n\circ \cdots\circ h_1.\\
	g_n= \psi_n^{-1}\circ f\circ \psi_n.
	\end{align}
	Then, we apply Lemma~\ref{lem:conv}, 
	to obtain a homeomorphism $g$ and a semi-conjugacy $\psi$, as  the limits
	\begin{align}
	g: =\lim_{n\to\infty} g_n.\\
	\psi\colon= \lim_{n\to \infty}\psi_n.\label{psin_psi}
	\end{align}
	Observe that the set $A=\cap_nU_{k_n}$ is continuum, as intersection of decreasing family of continua, all of diameter at least $\gamma$. As we announced in the construction, all sets $U_{k_n}$ can be viewed as $2^{-n}$-chains, so $A$ is chainable, 
	and each $U_{k_{n+1}}$ is crooked inside $U_{k_n}$, so $A$ is hereditarily indecomposable, therefore a pseudoarc (see \cite{Bing}).
	
	Now we apply Lemma~\ref{lem:fibres},
	obtaining that
	\begin{align}
	\psi^{-1}(y) \text{ is a singleton for any } y\notin \overline{\{f^n(p_\ast)\}_{n\in \Bbb Z}}. \label{singleteon}\\
	\psi^{-1}(p_\ast)=\bigcap_{n\geq 0}\psi_n^{-1}(\overline{V_{s_n}})=\bigcap_{n\geq 0} U_{k_n}=A. \label{Ais_collapsed}
	\end{align}
	The proof of Lemma~\ref{semiconjugacy} is completed. 
\end{proof}

\subsection{Properties of ``blow up'' of $P$.}\label{Q_is_pseudo}

In this subsection, we prove the following. 
\begin{theorem}\label{Qis_pseudocircle}
Let $\psi$ denote the semi-conjugacy map $\Bbb A$ obtained by construction in Lemma~\ref{semiconjugacy}, 
and denote $Q=\psi^{-1}(P)$.
Then $Q$ is a pseudo-circle.
\end{theorem}
Before we can prove the above result, we again need some preparation.

Recall that by Lemma~\ref{semiconjugacy} all preimages 
$\psi^{-1}(x)$ for $x\in \Bbb A$
are connected, 
so $\psi$ is a monotone map. 
It follows immediately 
that $Q$ is a continuum.
\begin{lemma}\label{nowheredense} 
The pseudo-arc $A=\psi^{-1}(p_\ast)\subset Q$
is nowhere dense in $Q$.
\end{lemma}

\begin{proof}
Suppose to the contrary that 
$A$ contains interior in $Q$. Equivalently, 
there exists an open subset $W\subset S^1\times (0,1)$
such that $W\cap Q \neq \emptyset$ and $W \cap Q=W \cap A$. 
Now we consider some connected component $U$
of $W\backslash Q$, which is open.
We can choose an 
arc $\alpha\colon [0,1] \to W$ 
connecting some point in $U$ to 
$\partial U\cap Q$. More precisely, 
we can define $\alpha: [0,1]\to W$, 
such that,  
$\alpha(0) \in A$ 
and $\alpha\big((0,1] \big) 
\subset U$. In particular, 
$\alpha\big( (0,1] \big) \cap Q= \emptyset$.
Then $\psi \big|_{ \alpha([0,1])}$ is a homeomorphism, and by 
the construction of $\psi$ (see Lemma~\ref{semiconjugacy}),
$\psi \circ \alpha(0)=p_\ast$, 
while $\psi \circ \alpha \big( (0,1] \big) 
\cap P = \emptyset$. However, 
$\psi\circ \alpha: [0,1] \to \psi(W)$ 
is an arc containing both
$p_\ast$ and some point in the complement of $P$. 
Since $p_\ast$ is contained in an internal composant of 
$P$,  
$\psi \circ \alpha \big([0,1] \big)$ 
intersects all composants of $P$. 
In particular it contains many points of $P$, which is 
a contradiction. The proof of the lemma is completed.
\end{proof}

\begin{lemma}\label{setof_injectivity} 
There exists a dense $G_\delta$ 
subset $I \subset Q$,
such that, 
for every point $x\in I$, 
\begin{equation}
\psi^{-1}\circ \psi(x) =\{x\}.
\end{equation}
\end{lemma}

\begin{proof}
Note that $f\big |_Q :Q\to Q$ is a homeomorphism, and by Lemma~\ref{nowheredense}, 
$A$ is nowhere dense in $Q$. Thus, $f^n(A)$ is nowhere dense for all $n\in \Bbb Z$.
Consider the set $\bigcup_{n\in \Bbb Z}f^n(A)$, which is a countable union of nowhere 
dense closed sets. Denote the complement by 
$I=  Q\backslash \bigcup_{n\in \Bbb Z}f^n(A)$.
Then by the Baire category theorem, $I$~as an 
intersection of a sequence of 
open and dense sets is a dense $G_\delta$ set.  
Moreover, it follows from (\ref{singleteon}) and (\ref{Ais_collapsed}) 
that, for any $x\in I$, 
\begin{equation}
\psi^{-1}\circ \psi(x) =\{x\}.
\end{equation}
The proof is completed.
\end{proof}

\begin{lemma}\label{indecomposable} 
$Q$ is an indecomposable continuum. 
\end{lemma}

\begin{proof} 
Suppose to the contrary that $Q$ is decomposable. Then we can write $Q=K\bigcup L$,
where $K$ and $L$ are proper subcontinua of $Q$. It follows that 
$\psi(K)$ and $\psi(L)$ are two subcontinua of $P$ as well as $P=\psi(K)\bigcup \psi(L)$.
Since $P$ is indecomposible, it follows either $\psi(K)$ or $\psi(L)$
is equal to $P$. Suppose without loss of generality that 
$\psi(L)=P$.
Now recall that $K \backslash L= Q\backslash L\neq \emptyset$ 
is open in $Q$. 
So by Lemma~\ref{setof_injectivity}, there exists 
$x\in  \big( K \backslash L\big) \cap I$, where 
$I=\big( Q\backslash \bigcup_{n\in \Bbb Z}f^n(A)\big)$.
In particular,  $\psi^{-1} \circ \psi(x)= \{x\}$. 
On the other hand, by assumption that $\psi(L)=P$,
there always exists $y\in L$ such that $\psi(y)=x$.
This 
provides a contradiction. Thus, $Q$ is indecomposable. 
\end{proof}

\begin{lemma}\label{hereditarily_indecomposable} 
The continuum $Q$ is hereditarily indecomposable. 
\end{lemma}

\begin{proof}
By comparing definitions, let us 
first point out the fact that,
the composants of $Q$ 
are precisely the preimages of the composants of $P$ under the semi-conjugacy 
$\psi$.

Suppose to the contrary that 
a non-degenerate 
subcontinuum $M\subset Q$ is decomposable. 
By Lemma~\ref{indecomposable} $M$ is proper, hence
$Q\backslash M$ is open in $Q$, and 
by Lemma~\ref{setof_injectivity},
$(Q \backslash M)\cap I\neq \emptyset$, 
where $I=Q\backslash \bigcup_{n\in \Bbb Z}f^n(A)$, 
restricted to which,
 the semi-conjugacy $\psi$
is $1$ to $1$. In particular $\psi(M)\neq P$, so $\psi(M)$ 
is a proper subcontinuum of $P$. 
It follows that, $\psi(M)$ is either a singleton or a pseudo-arc.
It is not hard to see that $\psi(M)$ can not be a singleton, 
because otherwise, by the construction of $\psi$,
$M\subset f^n(A)$ for some $n\in \Bbb Z$. Then, 
as a non-degenerate subcontinuum of the pseudo-arc 
$f^n(A)$, by Remark~\ref{pseudo_arc},
$M$ itself has to be a pseudo-arc, which is indecomposable. 

Now we suppose $\psi(M)$ is a pseudo-arc. 
By assumption, 
there are proper subcontinua $K$ and $L$ of $M$,
 such that $K \cup L = M$. 
 It follows that $\psi(K)$ and $\psi(L)$ are two subcontinua 
 of $\psi(M)$, which in turn is a subcontinuum of $P$. They satisfy 
\begin{equation}\label{psiKLM}
\psi(K) \cup \psi(L)= \psi (M).
\end{equation} 
Since $\psi(M)$ is hereditarily indecomposible, either $\psi(L)$
or $\psi(K)$ is equal to $\psi(M)$. 
Let us assume $\psi(L)=\psi(M)$ for definiteness. Then, 
for any $x\in K\backslash L$, there exists some $y\in L$, 
such that $\psi(x)=\psi(y)$. It follows that $x\notin I$, 
i.e., there exists some $n\in \Bbb Z$ such that $x\in f^n(A)$.

Lemma~\ref{indecomposable} tells us $Q$ is indecomposable. 
Now by Lemma~\ref{lemma_of_composant}, 
$f\big|_Q$ is fixed point free, which implies that no two members of the sequence $\{f^n(A)\}_{n\in \Bbb Z}$  of pseudo-arcs
  belong to the same composant of $Q$.
 Thus, 
$K\backslash L$ is contained in $f^n(A)$ for some $n\in \Bbb Z$. Up to 
composing by $f^{-n}$ if necessary, we simply assume that 
$K\backslash L\subset A$.
Since $K \backslash L=M\backslash L$ 
is open in $M$, there is some nonempty open set $W \subset S^1\times (0,1)$
such that,
\begin{align}
W \cap (K \backslash L) & = W \cap M \neq \emptyset.\\
W \cap L & = \emptyset. \label{capLnull}
\end{align} 
Since $\psi(K \backslash L)=p_\ast$, 
as in the proof of Lemma~\ref{nowheredense}, 
we can find some arc 
$\beta : [0,1] \to W$ such that 
\begin{align}
\beta(0) \in A.\\
\beta \big((0,1] \big) \cap M=\emptyset.
\end{align}
Since $\psi(M)$ is a pseudo-arc in $P$ containing 
the point $p_\ast$, 
for each $i\in \Bbb Z$,
it follows 
 $A_i\cap \psi(M)\neq \emptyset$. Since the sets $A_i$ are pseudo-arcs, it follows 
 either $A_i\subset \psi(M)$ or $\psi(M)\subset A_i$. 
 Since by Corollary~\ref{shrinking_diameter},
 $\lim_{n\to \infty}\text{diam} (A_n)=0$, 
 there is $k>0$ such that $A_k \subset \psi(M)$. 
 Now, by Lemma~\ref{Ap},
 $A_{k+1} \subset A_k$ contains 
 $p_\ast$, and 
 $A_{k+1}$ is in an internal composant of $A_k$. 
 
 Thus, denote $\beta' :=\psi\circ \beta\big( [0,1] \big)$
 which is an arc in $\Bbb A$ containing the point $p_\ast$.
 Since $\beta'$ 
 is in an internal composant of $A_k$,
 it must intersect all the composants of 
 $A_k$. 
 Then there exists $y \in A_k \cap \beta'$, with 
 $y \neq p_\ast$. 
 But $y \in \psi(M)$, 
 and there is some $x \in M$ such that $\psi(x)=y$.
 Then $x \in L$, 
because otherwise $\psi(x)=p_\ast$. 
But then $x\in W\cap L$ which is a contradiction to (\ref{capLnull}).
So $Q$ is hereditarily indecomposable, and the proof is completed. 
\end{proof}

\begin{lemma}\label{circle_like} The continuum $Q$ is circle-like and separates $\Ann$.
\end{lemma}

\begin{proof}
By the construction, denote by   
$\{\mathcal{C}_k\}_{k\geq 0}$  
for the defining sequence of circular chains for $P$, and let 
$\{h_n\}_{n\geq 0}$ and $\{\psi_n\}_{n\geq 0}$
be the sequence of homeomorphisms 
constructed in Lemma~\ref{semiconjugacy}.
Since the sequence $\{\mesh\big( \mathcal{C}_k \big)\}_{k\geq 0}$ 
converges to $0$, 
we can choose an increasing integer subsequence $\{k_j\}_{j\geq 0}$,
such that $k_0=0$ and for any $j\geq 0$ we have
\begin{align}
h_{j+1}^{-1} \big( \overline{ \rs (\mathcal{C}_{k_{j+1}}) }\big) & \subset 
 \rs\big( \mathcal{C}_{k_{j}} \big). \label{Dn_are_decreasing}\\
\mesh \big( \psi_{j+1}^{-1} ( \mathcal{C}_{k_{j+1}} ) \big) & < 2^{-j}.\label{mesh_is_small}
\end{align}
Now we denote $\mathcal D_j = \mathcal{C}_{k_j}$ for all $j\geq 0$.
For simplicity of the notation we will write $\psi_n$ instead of $\psi_{k_n}$.
Clearly, $\{\psi_n^{-1} (\mathcal D_n)\}_{n\geq 0}$ 
is a sequence of circular chains, and (\ref{Dn_are_decreasing}) show inductively
this sequence is in fact  
decreasing. By (\ref{mesh_is_small}), 
it follows that the continuum  
$\bigcap_{n\geq 0} \rs(\mathcal D_n)$
is circle-like. 
So we are left to show the following. 
\begin{equation}
Q=\bigcap_{n=0}^{\infty} \psi_n^{-1}(\rs ( \mathcal{D}_n)).
\end{equation}

Choose any $x \in\bigcap_{n=0}^{\infty} \psi_n^{-1}(\rs (\mathcal{D}_n))$. 
By the choice, 
$\psi_n (x) \in \rs (\mathcal{D}_n)$ for all  $n\geq 0$. 
Noting (\ref{psin_psi}) and \eqref{cn:subsets}, 
it follows that 
$\psi(x) \in \rs ( \mathcal{D}_n)$ for each $n\geq 0$. Thus, 
\begin{equation}
\psi(x) \in \bigcap_{n=0}^{\infty} \rs \big( \mathcal{D}_n \big)=P. 
\end{equation}
Then $x \in \psi^{-1}\circ \psi (x) \subset \psi^{-1}(P)=Q$ 
and so $\bigcap_{n=0}^{\infty} \psi^{-1}(\rs ( \mathcal{D}_n)) \subseteq Q$.

Next, suppose $Q'=\bigcap_{n=0}^{\infty} \psi_n^{-1}(\rs (\mathcal{D}_n))$ 
is a proper subcontinuum of $Q$.
By definition
$\psi_n(p_\ast)=p_\ast$ for all $n\geq 1$ and therefore
$\psi_n^{-1}(p_\ast)=p_\ast \in  \psi_n^{-1}(\rs ( \mathcal{D}_n))$ 
for each $n\geq 0$, which yields that $p_\ast \in Q'$. 
Then
$Q'$ as a proper subcontinuum of $Q$, 
must be contained in the 
composant $C(p_\ast)$
of $Q$ that contains $p_\ast$. 
Then $\psi(Q')$ is a proper subcontinuum of $P$. 
So, $\psi(Q')$ is either a pseudo-arc or a singleton. In particular, 
$\psi(Q')$ does not separate $\Bbb A$.

On the other hand, we know $\overline{\rs(\mathcal D_0)}$ separates $\Bbb A$. 
More precisely, 
the set $\Bbb A \backslash \overline {\rs ( \mathcal{D}_0)}$ 
has two connected components, 
and we call 
them $S$ and $T$, respectively. Since $\psi(Q')$ is not separating, there exists an arc 
$\beta :[0,1] \to \Bbb A$ such that $\beta(0)\in \psi(S)$ and $\beta(1)\in \psi(T)$, and 
$\beta( [0,1])\bigcap \psi(Q')=\emptyset$. 
Let $L:=\psi^{-1}(\beta([0,1]))$. Then $L$ is a continuum contained in $S^1\times (0,1)$. 
Moreover, 
\begin{equation}\label{AQ'empty}
L \cap Q' = \emptyset.
\end{equation} 
However, 
 $L \bigcap \psi_n^{-1} \big( \overline { \rs ( \mathcal{D}_n)} \big) \neq \emptyset$ for any  $n\geq 0$, since it intersects both $S$ and $T$. 
 Then for any $n\geq 0$, we 
 choose $y_n \in L \bigcap  \psi_n^{-1}  \big( \overline{ \rs ( \mathcal{D}_n)} \big)$. 
 By extracting a subsequence if necessary, 
 we can assume the sequence $\{ y_n\}_{n\geq 0}$ 
 converges to a point $y_\ast \in L$.
 
 We can also deduce from (\ref{Dn_are_decreasing}) 
the following for all $n >k \geq 1$.
\begin{align}
  \psi_n^{-1}( \overline { \rs ( \mathcal{D}_n) })
= &  \psi_{n-1}^{-1}h_n^{-1}\big( \overline {\rs ( \mathcal{D}_n)} \big) 
\\
\subset &
\psi_{n-1}^{-1}( \rs ( \mathcal{D}_{n-1}))\subset \ldots  \nonumber \\
\subset &  \psi_k^{-1}( \rs ( \mathcal{D}_{k})) . \nonumber
\end{align}
Therefore $y_n \in \psi_k^{-1} \big( \overline {\rs \big( \mathcal D_k\big)} \big)$ 
for any $n\geq k\geq 0$. It follows that 
$y\in \psi_k^{-1} \big( \overline {\rs \big( \mathcal D_k\big)} \big)$ for all $n\geq 0$. 
Thus $y\in L\bigcap Q'$, which is a contradiction to (\ref{AQ'empty}). This contradiction 
shows $Q'=Q$. We have completed the proof of the lemma.
\end{proof}

\begin{proof}[End of Proof of Proposition~\ref{Qis_pseudocircle}] 
It follows directly from Lemma~\ref{hereditarily_indecomposable}, Lemma~\ref{circle_like} and the characterization stated 
in
Remark~\ref{pseudo_circle_characterization}, that $Q$ is a pseudo-circle.
\end{proof}

\begin{cor}\label{shrinking_diameter} 
Let $f\colon \Bbb A\to \Bbb A$ be the homeomorphism given in Lemma~\ref{Handel} with semiconjugacy $\psi\colon (\Bbb A,g)\to (\Bbb A,f)$, 
mapping the pseudo-circle $Q$ onto the pseudo-circle $P$
provided by Lemma~\ref{semiconjugacy}. 
Then, for the pseudo-arc $A=\psi^{-1}(p_\ast)$ we have 
\begin{equation}
\lim_{|n|\to \infty} \text{diam} \big(g^n(A) \big)=0.
\end{equation} 
\end{cor}

\begin{proof}
Note that the semi-conjugacy $\psi$ on $\Bbb A$ can be regarded as 
a continuous map from $Q$ to $P$. 
Suppose to the contrary
that for some integer subsequence $\{n_k\}_{k\geq 0}$ and some $\delta>0$, 
every pseudo-arc 
$f^{n_k}(A)$ has diameter at least $\delta$.
Let us fix 
some point $x\in Q$ such that 
\begin{equation}\label{x_isinjective}
\psi^{-1} \circ \psi(x)=\{x\}.
\end{equation}
By minimality of $f$ on $P$, 
we can extract a further 
subsequence if necessary, still called $\{n_k\}_{k\geq 0}$, 
such that $g^{n_k}(A)$ converges to some set $K$ containing $x$.
By definition $g^n(A)=\psi^{-1}(f^n(p_*))$ and so
by Lemma~\ref{usc}, $\psi^{-1}\circ \psi(x)\supset K$. However, $K$ 
has diameter at least $\delta$, which 
is a contradiction with (\ref{x_isinjective}).
\end{proof}

\begin{proof}[Proof of Theorem~\ref{main_section}]
We have already proved in Theorem~\ref{Qis_pseudocircle}
that $Q$ is a pseudo-circle. 
By Lemma~\ref{nowheredense}, 
any iterate $g^n(A)$ 
has empty interior in $Q$ for any $n\in \Bbb Z$. 
Thus, the set 
$Q\backslash \bigcap_{n\in \Bbb Z}g^n(A)$
is a dense $G_\delta$ subset of $Q$. This means that 
the set of points $y \in Q$ such that $\psi^{-1}\circ \psi(y)=\{y\}$
is a dense $G_\delta$-set. In particular, $\psi$
is an almost $1$-to-$1$ semi-conjugacy between 
$(Q,g)$ and $(P,f)$. Therefore, it follows immediately 
that $(Q,g)$ is minimal since $(P,f)$ is minimal. (see for example~\cite{Glasner}).
Equation (\ref{shrinking}) was established in Corollary~\ref{shrinking_diameter}. 
The proof is completed.
\end{proof}

\section{Proof of the Main Theorem}\label{final}
In this section we finish the proof of our main result. It seems that it is known that a monotone image of the pseudo-circle is homeomorphic to it (see e.g. p. 91, \cite{Lewis}), but for completness we provide a short justification below. 

\begin{prop}
	\label{pseudocircle_quotient}
	Suppose $\phi:  \Bbb A \to  \Bbb A$ 
	is a monotone continuous
	surjective map. 
	Assume there is a non-degenerate 
	continuum $Y\subset \Bbb A$, such that $\phi^{-1}(Y)=P$. Then $Y$ is a pseudo-circle.
\end{prop}

\begin{proof} 
	We need to show that $Y$ is circle-like, plane separating and hereditarily indecomposable. 
	
	It is well known that a monotone image of a circle-like continuum is circle-like (see e.g. Lemma 9, \cite{Krupski})
	Moreover $Y$ is non-degenerate and $\phi^{-1}(Y)= P$ is plane separating. 
	By the continuity of $\phi$, 
	it is immediate that $Y$ is also plane separating. It remains to show that $Y$ is hereditarily indecomposable. But this is also immediate, since if a subcontinuum $K$ of $Y$ were decomposable, then $\phi^{-1}(K)$ would also be a decomposable continuum, resulting in a contradiction. This shows that $Y$ is a pseudo-circle. 
\end{proof}

Theorem~\ref{main} 
follows directly 
from the following statement. 
\begin{theorem}
There exists a continuous surjection 
$g'\colon \Ann\to \Ann$ with an invariant pseudo-circle 
$P\subset \Ann$ such that $(g',P)$ is minimal but is not one-to-one.
\end{theorem}

\begin{proof}
Let
$g\colon \Ann\to \Ann$ denote 
the homeomorphism provided by 
Proposition~\ref{main_section} minimal on pseudo-circle $Q$ and let $A\subset Q$ be the associated pseudo-arc. 
Define the equivalence relation 
$\sim$ in $\Ann$, such that, 
$x\sim y$ if and only if either 
$x=y$, or $x,y\in g^n(A)$ for some $n\geq 0$. 
Since by Proposition~\ref{main_section}, 
\begin{equation}
\lim_{|n|\to \infty} \text{diam} \big( g^n(A)\big) =0,
\end{equation}
it follows that 
$\sim$ is a closed equivalence relation with 
connected equivalence classes. Moreover, 
it is not hard to see the family of equivalence classes form a 
upper semi-continuous family. 
Thus, by Moore's theorem (see Lemma~\ref{Moore_theorem}), 
there exists a semi-conjugacy
$\pi\colon \Bbb A\to \Bbb A$, which collapses each equivalence 
class into one point in $\Bbb A$. 
Then it induces a map $\hat{g}\colon \Bbb A\to \Bbb A$, such that 
$\pi\circ g= \hat{g} \circ \pi$. 
Note that $\hat{g}$ 
is continuous non-invertible surjection
and $\pi \colon (\Ann,g)\to (\Ann,\hat{g})$ is a monotone surjective continuous 
map (a factor map). Clearly $\hat{g}$ must be minimal.
Then, by Proposition~\ref{pseudocircle_quotient}, the proof is completed. 
\end{proof} 

\bibliographystyle{plain}

\end{document}